\documentclass{amsart}

\usepackage{amssymb,amsmath}
\usepackage[all]{xy}
\usepackage{graphicx}

\newtheorem{theo}{Theorem}[section]
\newtheorem{lemm}{Lemma}[section]

\newtheorem{rema}{Remark}[section]
\newtheorem{exem}{Example}[section]

\def\hls#1{{\rm HLS}(#1)}
\def\HOS#1{{\rm HOS}(#1)}
\def\dist{{\rm dist}}

\title[Hausdorff Leaf Spaces]{Hausdorff Leaf Spaces for codim-1 foliations}
\author{Szymon M. Walczak}
\date{version \today}

\begin{document}

\begin{abstract}
The topology of the Hausdorff leaf spaces (briefly the HLS) for a codim-1 foliation is the main topic of this paper. At the beginning, the connection between the Hausdorff leaf space and a warped foliations is examined.  Next, the author describes the HLS for all basic constructions of foliations such as transverse and tangential gluing, spinning, turbulization, and suspension. Finally, it is shown that the HLS for any codim-1 foliation on a compact Riemannian manifold is isometric to a finite connected metric graph. In addition, the author proves that for any finite connected metric graph $G$ there exists a compact foliated Riemannian manifold $(M,\mathcal{F},g)$ with codim-1 foliation such that the HLS for $\mathcal{F}$ is isometric to $G$. Finally, the necessary and sufficient condition for warped foliations of codim-1 to converge to $\hls{\mathcal{F}}$ is given.
\end{abstract}

\maketitle

\section{Introduction}\label{Introduction}

In the 70-s M. Berger has presented the concept of modification of a Riemannian metric of $S^3$ along the fibers of the Hopf fibration. Following this concept, the author of this paper has introduced the notion of warped foliation \cite{W3}. Later on, the author has examined the limits of a sequence of warped compact foliations \cite{W} and has proposed the notion of the Hausdorff leaf space (briefly the HLS) for a foliation on a compact Riemannian manifold. 

This paper is the continuation of the research held in \cite{W}. At the beginning, the author shows that the HLS for any foliation $\mathcal{F}$ on a compact Riemannian manifold $(M,g)$ is the Gromov-Hausdorff limit of a sequence of warped foliations with warping functions converging to zero on a dense subset $G\subset M$ (Section \ref{Convergence}, Theorem \ref{ConvergenceTheorem}). Next, he examines the Hausdorff leaf spaces for all natural constructions of the foliation listed in \cite{CC}. Namely, the HLS for tangential and transverse gluing, spinning, turbulization, and suspension are studied (Section \ref{BasicContructions}).

The main results of this paper are developed in Section \ref{MainResults} (Theorem \ref{MainTheorem1} and Theorem \ref{MainTheorem2}), where the complete description of the Hausdorff leaf space for a codim-1 foliation on a compact Riemannian manifold is presented. It is shown that the HLS for a codim-1 foliation is isometric to a finite connected metric graph, while for every finite connected metric graph $G$ there exists a foliated Riemannian manifold $(M,\mathcal{F},g)$ such that the Hausdorff leaf space for $\mathcal{F}$ is isometric to $G$. Finally (Theorem \ref{SufficientNecessaryCodim1}), the necessary and sufficient condition for the sequence $(f_n)$ of warping function on a compact Riemannian manifold carrying foliation of codim-1 to have a sequence of warped foliations $(M_{f_n})$ converging to the Hausdorff leaf space for the foliation $\mathcal{F}$ is shown.

For the theory of foliations we refer to \cite{CC} or \cite{HH}.

\section{Preliminaries}\label{Preliminaries}

\subsection{Hausdorff leaf spaces}\label{Preliminaries-HLS}

Let us recall the notion of Hausdorff leaf space \cite{W}:

Let $(M,\mathcal{F},g)$ be a compact foliated manifold. Let us set
\[
\rho(L,L') = \inf \{\sum_{i=1}^{n-1} {\rm dist} (L_i,L_{i+1})\},
\]
where the infimum is taken over all finite sequences of leaves beginning at $L_1=L$ and ending at $L_n=L'$ (Figure \ref{HLSFigure}). Let $\sim$ be an equivalence relation in the space of leaves $\mathcal{L}$ defined by:
\begin{equation*}
L\sim L' \Leftrightarrow \rho(L,L')=0,\quad L,L'\in\mathcal{L}.
\end{equation*}
Let $\tilde{\mathcal{L}} = \mathcal{L}/_{\sim}$. Put 
\[
\tilde\rho([L],[L']) = \rho(L,L'), 
\]
where $[L],[L']\in\tilde{\mathcal{L}}$. $(\tilde{\mathcal{L}},\tilde\rho)$ is a metric space. We call it {\em the Hausdorff leaf space} for the foliation $\mathcal{F}$ (briefly the HLS), and we denote it by $\hls{\mathcal{F}}$.

\begin{figure}
\includegraphics[scale=0.5]{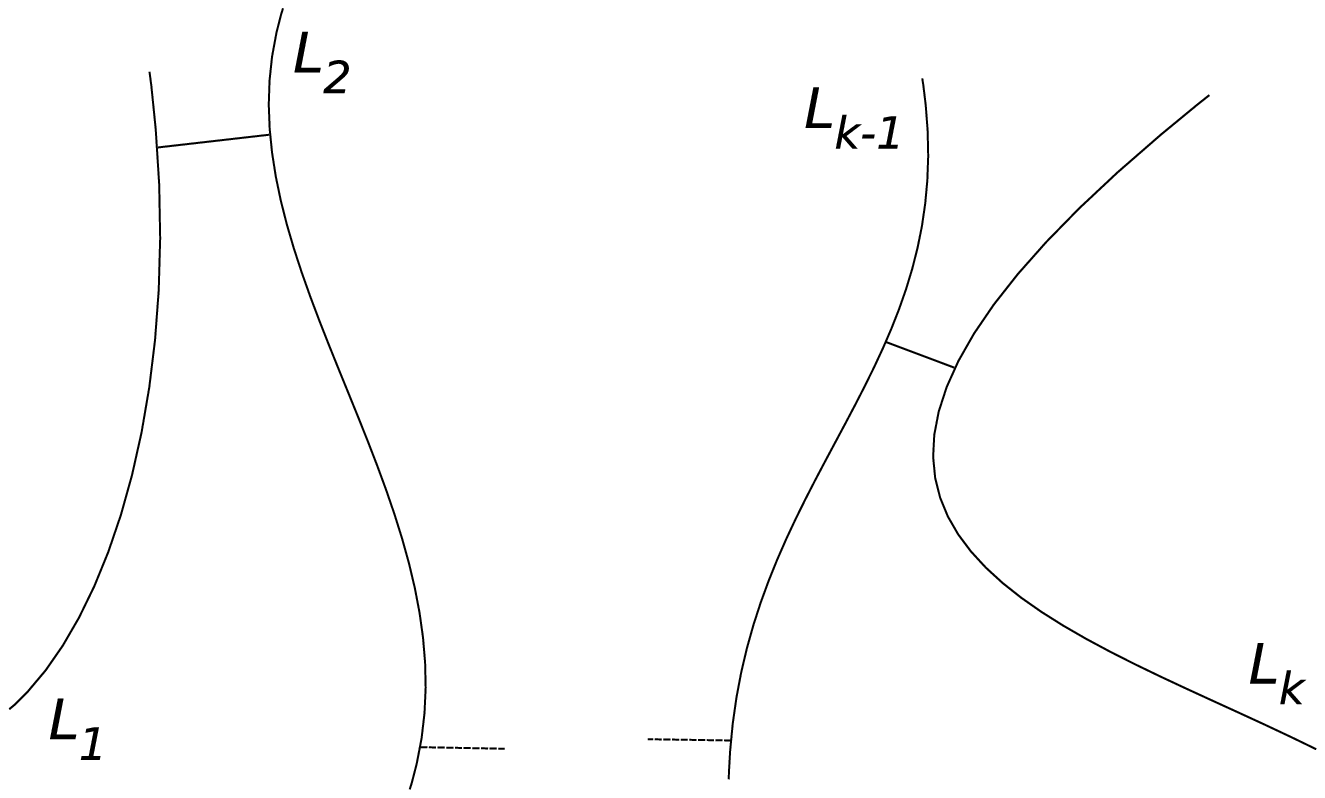} 
\caption{The idea of $\rho$.}
\label{HLSFigure}
\end{figure}
\begin{rema}
Equivalently, the Hausdorff leaf space can be defined as follows:

Following \cite{BBI}, one can define in a metric space $(X,d)$ equipped with an equivalence relation $R$ the {\it quotient pseudo-metric} $d_R$ as
\[
d_R(x,y) = \inf \{\sum_{i=1}^k d(p_i,q_i): p_1=x, q_k=y, k\in\mathbb{N}\}.
\]
where the infimum is taken over all sequences $\{pi\}_{1\leq i\leq N}$, $\{q_i\}_{1\leq i\leq N}$, $N\in\mathbb{N}$, such that 
\[
(p_{i+1},q_i)\in R.
\]
Consider a metric space $(X/R,d_R)$ and identify such points for which $d_R$ is equal to zero. Obtained metric space is called the {\it quotient metric space}. 

Let $(M,\mathcal{F},g)$ be a compact foliated Riemannian manifold, and let $R$ be the relation of belonging to the same leaf of $\mathcal{F}$. Using $R$ in $M$ we get the alternative definition.
\end{rema}

\begin{rema}
Let $\mathcal{F}$ be a codim-1 foliation on a compact Riemannian manifold $(M,g)$. One can define in the space of leaves a relation as follows: a leaf $L$ is related to a leaf $L'$ iff $L$ is contained in a closure ${\rm cl}L'$ of a leaf $L'$. This relation defines an equivalence relation $\equiv$ in the leaf space $\mathcal{L}$. One can check that equipping $\mathcal{L}/_{\equiv}$ with quotient metric one obtains (for a foliation of codimension one) Hausdorff leaf space for the foliation $\mathcal{F}$.
\end{rema}

\begin{lemm}
For every foliation $\mathcal{F}$ on a compact foliated Riemannian manifold the $\hls{\mathcal{F}}$ is a length space.
\end{lemm}
\begin{proof}
By the definition of the length metric \cite{G}, for every two points $x,y\in\hls{\mathcal{F}}$ and any curve $c:[0,1]\to\hls{\mathcal{F}}$ such that $c(0)=x$, $c(1)=y$ we have $\tilde{\rho}(x,y)\leq l(c)$. The opposite inequality follows directly from the definition of the HLS. 
\end{proof}

\subsection{Gluing metric spaces}\label{Preliminaries-Gluing}

Following \cite{BBI}, we now describe how to glue length spaces:

 Let $(X_{\alpha},d_{\alpha})$ be a family of length spaces. Set the length metric $d$ on a disjoint union $\amalg_{\alpha} X_{\alpha}$ as follows:

If $x,y\in X_{\alpha}$, then $d(x,y)= d_{\alpha}(x,y)$; Otherwise, set $d(x,y)=\infty$. The metric $d$ is called the {\em length metric of disjoint union}. 

Now, let $(X,d_X)$ and $(Y,d_Y)$ be two length spaces, while $f:A\to B$ be a bijection between two subsets $A\subset X$ and $B\subset Y$. Equip $Z=X\amalg Y$ with the length metric of disjoint union. Introduce the equivalence relation $\sim$ as the smallest equivalence relation containing relation generated by the relation $x\sim y$ iff $f(x)=y$. The result of gluing $X$ and $Y$ along $f$ is the metric space $(Z/_{\sim},d_{\sim})$.

\subsection{Warped foliations}\label{Preliminaries-WarpedFoliations}

We recall here the notion of warped foliation \cite{W}. The Hausdorff leaf space for warped foliation will be the main topic of our interest in Section 2. Moreover, the results of Section 2 will be used as a tool in Sections 3 and 4.

Let $(M,\mathcal{F},g)$ be a foliated Riemannian manifold and $f:M\to (0,\infty)$ be a basic function on $M$, i.e. a function constant along the leaves of $\mathcal{F}$. We modify the Riemannian structure $g$ to $g_f$ in the following way: $g_f (v,w) = f^2 g(v,w)$ while both $v,w$ are tangent to the foliation $\mathcal{F}$, but if at least one of vectors $v,w$ is perpendicular to $\mathcal{F}$ then we set $g_f (v,w) = g(v,w)$. Foliated Riemannian manifold $(M,\mathcal{F},g_f)$ is called here \textit{the warped foliation} and denoted by $M_f$. The function $f$ is called \textit{the warping function}.

\subsection{Gromov-Hausdorff convergence}\label{Preliminaries-GHConvergence}

Recall the notion of Gromov-Hausdorff convergence \cite{G}. Let $(X, d_X)$ and $(Y, d_Y)$ be an arbitrary compact metric spaces. The distance of $X$ and $Y$ can be defined as
\[
d_{GH}(X,Y):={\rm inf}\{d_H(X,Y)\},
\]
where $d$ ranges over all admissible metric on disjoint union $X\amalg Y$, i.e. $d$ is an extension of $d_X$ and $d_Y$, while $d_H$ denotes the Hausdorff distance. The number $d_{GH}(X,Y)$ is called {\it the Gromov-Hausdorff distance} of metric spaces $X$ and $Y$.

\begin{theo}\label{GromovDistanceTheorem}
$d_{GH}(X,Y)=0$ iff $(X,d_X)$ is isometric to $(Y,d_Y)$.\hfill $\square$
\end{theo}

The Gromov Lemma (below) will be used widely throughout his paper.

\begin{lemm}\label{GromovLemma}
Let $(X,d_X)$ and $(Y,d_Y)$ be arbitrary compact metric spaces, and let 
\begin{align*}
A=\{x_1,\dots ,x_k\}\subset X,\\
B=\{y_1,\dots ,y_k\}\subset Y
\end{align*}
be $\varepsilon$-nets satisfying  for all $1\leq i,j\leq k$ the condition 
\[
|d_X(x_i,x_j) - d_Y(y_i,y_j)|\leq \varepsilon.
\]
Then $d_{GH}(X,Y)\leq 3\varepsilon$.\hfill $\square$
\end{lemm}

\section{Convergence theorem}\label{Convergence}

Consider a sequence $(f_n)_{n\in\mathbb{N}}$, $f_n:M\to (0,\infty)$, of warping function on a compact foliated Riemannian manifold $(M,\mathcal{F},g)$. One can ask, how does the limit in Gromov-Hausdorff topology of a sequence of warped foliations $(M_{f_n})_{n\in\mathbb{N}}$ look like. Let $G\subset M$ be a dense subset. 

\begin{theo}\label{ConvergenceTheorem}
For an arbitrary compact foliated manifold $(M,\mathcal{F},g)$ and any sequence $(f_n)_{n\in\mathbb{N}}$, $f_n: M\to [0,1)$, of warping functions on $M$ converging to zero on $G$, the Gromov-Hausdorff limit of a sequence of warped foliations is isometric to $\hls{\mathcal{F}}$.
\end{theo}

Before we prove Theorem \ref{ConvergenceTheorem}, we give a simple definition and observe an important fact.

We say that two metric structures $g$ and $g'$ on a compact foliated Riemannian manifold $(M,\mathcal{F})$ {\it coincide on the orthogonal bundle} $\mathcal{F}^{\perp}$ if every vector $v$ perpendicular to $\mathcal{F}$ in $g$ is perpendicular in $g'$ and vice versa, and $g(v,w) = g'(v,w)$ for any vectors $v,w$ perpendicular to $\mathcal{F}$ either in $g$ or $g'$.

\begin{lemm}\label{ConvergenceTheoremLemma1}
Let $g$ and $g'$ be any Riemannian structures on $M$ which coincide on the orthogonal bundle $\mathcal{F}^{\perp}$. Then $\tilde\rho = \tilde{\rho'}$.
\end{lemm}
\begin{proof}
Since $M$ is compact, we can assume that $g\leq C\cdot g'$ for a certain constant $C\geq 1$. Let $\rho$ and $\rho'$ be pseudometrics given by
\begin{align*}
\rho(L,L') = \inf \{\sum_{i=1}^{n-1} {\rm dist} (L_i,L_{i+1})\}, \\
\rho'(L,L') = \inf \{\sum_{i=1}^{n-1} {\rm dist}' (L_i,L_{i+1})\},
\end{align*}
where $\dist$ and $\dist'$ denote the distance of the leaves in $g$ and $g'$, respectively. Since the geometry of $M$ is bounded, then for every $A>0$ and $\epsilon>0$ there exists $\delta>0$ such that for every smooth curve $\gamma:[0,l(\gamma)]\to M$ parametrized naturally satisfying
\begin{enumerate}
\item $\dot\gamma(0)$ is perpendicular to $\mathcal{F}$,
\item the $g'$-length $l'(\gamma)$ is smaller than $\delta$,
\item the $g'$-geodesic curvature $|k_g(\gamma)|$ is smaller than $A$,
\end{enumerate}
the $g$-length of the component tangent to $\mathcal{F}$ satisfies 
\begin{equation*}
|\dot\gamma^{\top}|<\epsilon.
\end{equation*}
Let $\epsilon>0$, $L,L'\in\mathcal{F}$ be such that $d={\rm dist}'(L,L')<\delta$. Let $\gamma:[0,l'(\gamma)]\to M$ be a curve such that its length in $g'$ satisfies $d\leq l'(\gamma)\leq \delta$. We have
\begin{align*}
{\rm dist}(L,L') \leq l(\gamma) = \int_{[0,l'(\gamma)]} |\dot\gamma| \leq \int_{[0,l'(\gamma)]} |\dot\gamma^{\top}| + \int_{[0,l'(\gamma)]} |\dot\gamma^{\bot}| \\
\leq C\cdot l'(\gamma)\cdot \epsilon + \int_{[0,l'(\gamma)]} |\dot\gamma^{\bot}|' \leq (1+C\epsilon)\cdot l'(\gamma).
\end{align*}
Since $\gamma$ was chosen arbitrarily, we conclude that
\begin{equation*}
{\rm dist}(L,L') \leq (1+C\epsilon)\cdot {\rm dist}'(L,L').
\end{equation*}
Now, for every sequence of leaves $L_1,\dots,L_n$ such that $L_1=L$, $L_n=L'$ and satisfying 
\begin{equation*}
\sum_{i=1}^{n-1} {\rm dist}' (L_i,L_{i+1}) \leq \rho'(L,L') + \epsilon,
\end{equation*}
and such that ${\rm dist}' (L_i,L_{i+1})<\delta$ for all $i\in\{1,\dots,n-1\}$, we obtain
\begin{align*}
\rho(L,L')\leq \sum_{i=1}^{n-1} {\rm dist} (L_i,L_{i+1})\leq (1+C\epsilon)\cdot(\sum_{i=1}^{n-1} {\rm dist}' (L_i,L_{i+1})) \\
\leq (1+C\epsilon)\cdot (\rho'(L,L')+\epsilon).
\end{align*}
Tending with $\epsilon$ to zero we get that $\rho\leq \rho'$. Consequently $\tilde\rho\leq \tilde\rho'$. Similarly, we can show that $\tilde\rho'\leq \tilde\rho$.
\end{proof}

We now turn to a proof of Theorem \ref{ConvergenceTheorem}. Denote by $\pi:M\to \hls{\mathcal{F}}$ a natural projection given by $\pi(x)=[L_x]_{\sim}$, where $\sim$ is the equivalence relation defined in Section 1.1.

\begin{proof} 
Let $\epsilon>0$ and let $\{x_1,\dots,x_k\}$ be an $\epsilon$-net on $M$ contained in $G$. Let $i,j\in\{1,\dots,k\}$. Choose a family of leaves $F^{ij} = \{F_0^{ij},\dots,F_{d_i^j}^{ij}\}$ such that  $L_{x_i}=F_0^{ij}$, $L_{x_j}=F_{d_i^j}^{ij}$, $F_{\nu}^{ij}\subset \pi^{-1}(\pi(G))$ for any $0\leq \nu \leq d_i^j$. Next, consider a family of curves 
\[
\gamma_0^{ij},\dots,\gamma_{d_i^j-1}^{ij}:[0,1]\to M
\]
satisfying $\gamma_{\nu}^{ij}(0)\in L_{x_{\nu}}$, $\gamma_{\nu}^{ij}(1)\in L_{x_{\nu+1}}$, and \begin{equation}\label{ConvergenceTheoremEquation1}
\sum_{\nu=0}^{d_i^j-1} l(\gamma_{\nu}^{ij}) \leq \tilde\rho(\pi(L_{x_i}),\pi(L_{x_j}))+\epsilon.
\end{equation}
Let $d=\max\{d_i^j\}$. Since $f_n\to 0$ on $G$, and the number of leaves involved in $F^{ij}$, $i,j=1,\dots,k$, is finite, there exists $N\in\mathbb{N}$ such that for any $n>N$, $i,j\in\{1,\dots,k\}$ and $\nu\in \{0,\dots, d_i^j-1\}$ we have
\begin{align}\label{ConvergenceTheoremEquation2}
d_{F_{\nu}^{ij}}(\gamma_{\nu}^{ij}(1),\gamma_{\nu+1}^{ij}(0)) \leq \frac{\epsilon}{d},\\
\nonumber
d_{F_{\nu}^{ij}} (x_i,\gamma_{0}^{ij}(0))\leq \frac{\epsilon}{d}, \textrm{ and } d_{F_{\nu}^{ij}} (\gamma_{d_i^j-1}^{ij}(1),x_j)\leq \frac{\epsilon}{d}.
\end{align}

Let us pick one point in each $\{x_1,\dots,x_k\}\cap\pi^{-1}(\pi(x_i))$, $i=1,\dots,k$. We obtain a set $\{y_1,\dots,y_m\}$ ($m\leq k$) with the property $\pi(y_i)\neq \pi(y_j)$ iff $i\neq j$.

Let $n>N$. Direct calculation shows that the points $y_1,\dots,y_m$ form a $3\epsilon$-net on $(M,g_{f_n})$. Moreover, by (\ref{ConvergenceTheoremEquation1}) and (\ref{ConvergenceTheoremEquation2}),
\begin{equation*}
d_n(y_i,y_j)\leq \tilde\rho(\pi(L_{y_i}),\pi(L_{y_j})) + 2\epsilon.
\end{equation*}
Next, by Lemma \ref{ConvergenceTheoremLemma1},
\begin{equation*} 
\tilde\rho([L_{y_i}],[L_{y_j}]) = \tilde\rho_n(\pi(L_{y_i}),\pi(L_{y_j})) \leq d_n(y_i,y_j).
\end{equation*}
The set $\pi(\{y_1,\dots,y_m\}) = \pi(\{x_1,\dots,x_k\})$ provides an $\epsilon$-net on $\hls{\mathcal{F}}$. By Lemma \ref{GromovLemma}, $d_{GH}(M_{f_n},\hls{\mathcal{F}})\leq 9\epsilon$. Tending with $\epsilon$ to zero we get that 
\[
d_{GH}(M_{f_n},\hls{\mathcal{F}})=0.
\]
 Theorem \ref{GromovDistanceTheorem} completes the proof.
\end{proof}

\section{Basic constructions}\label{BasicContructions}

Studying foliations one can learn that there are several basic constructions for building foliations \cite{CC}. In this chapter we examine the HLS for the following constructions: tangential and transverse gluing, two transverse modifications - turbulization and spinning along a transverse boundary component, and for suspension.

We only provide here a detailed proof for tangential gluing and turbulization, which are used in Section \ref{MainResults}. For transverse gluing and for spinning we give the characterization of their HLS's and we only provide an outline of a proof. Details, analogically as for tangential gluing and turbulization, are left to the reader.

\subsection{Tangential gluing}

Let us assume that $(M_i,\mathcal{F}_i,g_i)$, $i=1,2$, are compact foliated Riemannian manifolds with boundary, while $\mathcal{F}_i$ is a foliation tangent to the boundary. Let $S_i\subset \partial M_i$ ($i=1,2$) be a union of boundary components, and let $h:S_1\to S_2$ be an isometry mapping leaves onto leaves. According to \cite{CC}, identify $S_1$ with $S_2$ using $x\equiv h(x)$, and form the quotient foliated manifold $M=M_1\cup_h M_2$ with foliation $\mathcal{F}=\mathcal{F}_1\cup_h \mathcal{F}_2$ defined by the leaves of $\mathcal{F}_i$ (Figure \ref{TangentFigure1}). 

Let us assume that one can obtain a smooth Riemannian structure $g$ on $M$ with the property $g|M_i=g_i$ ($i=1,2$).

\begin{figure}[h]
\centering
 \includegraphics[scale=1]{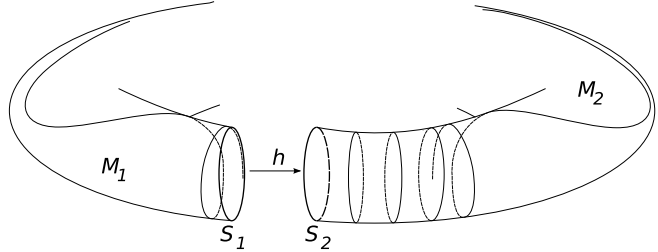}
 \caption{Tangential gluing.}
 \label{TangentFigure1}
\end{figure}

Denote by $\pi:M\to \hls{\mathcal{F}}$, $\pi_i:M_i\to\hls{\mathcal{F}_i}$ ($i=1,2$) natural projections. Consider the smallest equivalence relation $\sim$ in disjoint union 
\[
\hls{\mathcal{F}_1}\amalg \hls{\mathcal{F}_2}
\]
containing the relation defined as follows:
\[
\pi(L)\sim \pi(L') \Leftrightarrow \exists_{x\in\pi_1^{-1}(\pi(L))}\quad \pi_2(h(x))=\pi_2(L').
\]
Let $X=\hls{\mathcal{F}_1}\amalg \hls{\mathcal{F}_2}/_{\sim}$ endowed with quotient metric $d_X$.

Denote by $\Phi:\hls{\mathcal{F}}\amalg \hls{\mathcal{F}}\to X$, $\tilde \pi:M_1\amalg M_2\to \hls{\mathcal{F}}\amalg \hls{\mathcal{F}}$, and $p:M_1\amalg M_2\to M$ natural projections (see Figure \ref{TangentialFigure}). 

\begin{theo}\label{TangentialGluing}
The space $\hls{\mathcal{F}}$ is isometric to $(X,d_X)$.
\end{theo}

\begin{figure}[h]
$$\xymatrix{
M_1\amalg M_2 \ar[r]^{p} \ar[d]^{\tilde \pi}& M \ar[r]^{\pi} & \hls{\mathcal{F}}\\
\hls{\mathcal{F}_1} \amalg \hls{\mathcal{F}_2} \ar[r]_{\hskip50pt \Phi} & X
}$$
\caption{The projections for Theorem \ref{TangentialGluing}.}
\label{TangentialFigure}
\end{figure}

We begin the proof by a following:

\begin{lemm}\label{TangentialGluingLemma}
For any $x,y\in M_1\amalg M_2$ 
\[
d_X(\Phi(\tilde{\pi}(x)),\Phi(\tilde{\pi}(y))) = \tilde{\rho}(\pi(p(x)),\pi(p(y))).
\]
\end{lemm}
\begin{proof}
Let $\epsilon>0$. Consider points $\pi(p(x))$ and $\pi(p(y))$. By the definition of $\hls{\mathcal{F}}$, there exist points $r_1,q_1,\dots,r_k,q_k$ in the disjoint union $M_1\amalg M_2$ such that $r_{\nu},q_{\nu}$ are the points in the component $M_1$ or $M_2$, and $r_1\in L_x$, $q_k\in L_y$ ($L_x$ and $L_y$ are here the leaves of the appropriate foliation $\mathcal{F}_1$ or $\mathcal{F}_2$). Moreover, $p(q_{\nu})$ and $p(r_{\nu+1})$ lie in the same leaf of $\mathcal{F}$, and
\[
\sum_{\nu=1}^{k} \bar d(r_{\nu},q_{\nu}) \leq \tilde{\rho}(\pi(p(x)),\pi(p(y))) + \epsilon,
\]
where $\bar d$ is the length metric of disjoint union in $M_1\amalg M_2$ (see Section 1.2). Denote by $\tilde d$ the length metric of disjoint union in $\hls{\mathcal{F}_1}\amalg \hls{\mathcal{F}_2}$. By the definition of $X$, we have
\begin{align*}
\sum_{\nu=1}^{k} \bar d(r_{\nu},q_{\nu}) \geq \sum_{\nu=1}^{k} \tilde{d}(\tilde{\pi}(r_{\nu}),\tilde{\pi}(q_{\nu}))\\
\geq d_X(\Phi(\tilde{\pi}(x)),\Phi(\tilde{\pi}(y))).
\end{align*}
Finally,
\begin{equation}\label{TangentilaGluingLm1Eq1}
d_X(\Phi(\tilde{\pi}(x)),\Phi(\tilde{\pi}(y)))\leq \tilde{\rho}(\pi(p(x)),\pi(p(y))) + \epsilon
\end{equation}

Next, consider points $\Phi(\tilde{\pi}(x))$ and $\Phi(\tilde{\pi}(y))$. There exist points $r_1,q_1,\dots,r_k,q_k$ in the disjoint union $\hls{\mathcal{F}_1}\amalg \hls{\mathcal{F}_2}$ such that $\Phi(q_{\nu})=\Phi(r_{\nu+1})$ ($\nu=1,\dots,k$), $\Phi(r_1) = \Phi(\tilde{\pi}(x))$, $\Phi(q_k) = \Phi(\tilde{\pi}(y))$, and
\[
\sum_{\nu=1}^{k} \tilde d(r_{\nu},q_{\nu}) \leq d_X(\Phi(\tilde{\pi}(x)),\Phi(\tilde{\pi}(y))),
\]
where $\tilde d$ denotes again the length metric of disjoint union in $\hls{\mathcal{F}_1}\amalg \hls{\mathcal{F}_2}$. Now, for every $\nu=1,\dots,k$ one can find a sequence of leaves $L^{\nu}_1,\dots,L^{\nu}_{l_{\nu}}$ of the appropriate foliation ($\mathcal{F}_1$ if $r_{\nu},q_{\nu}\in \hls{\mathcal{F}_1}$ or $\mathcal{F}_2$ if $r_{\nu},q_{\nu}\in \hls{\mathcal{F}_2}$) satisfying $r_{\nu}\in L^{\nu}_1$, $q_{\nu}\in L^{\nu}_{l_{\nu}}$, and
\[
\sum_{\mu=1}^{l_{\nu}-1} \dist(L^{\nu}_{\mu},L^{\nu}_{\mu+1}) \leq  \tilde d(r_{\nu},q_{\nu}) + \frac{\epsilon}{k}.
\]
Since $h$ maps leaves onto leaves, one can consider the leaves described above as leaves of a foliation $\mathcal{F}$. Moreover, $\tilde{\pi}(x)\in \tilde{\pi}(L^1_1)$, and $\tilde{\pi}(y)\in \tilde{\pi}(L^k_{l_{\nu}})$. Hence we have
\begin{equation}\label{TangentilaGluingLm1Eq2}
\tilde{\rho}(\pi(p(x)),\pi(p(y))) \leq  d_X(\Phi(\tilde{\pi}(x)),\Phi(\tilde{\pi}(y))) + 2\epsilon.
\end{equation}
Passing with $\epsilon$ to zero in inequalities (\ref{TangentilaGluingLm1Eq1}) and (\ref{TangentilaGluingLm1Eq2}), we get that
\[
\tilde{\rho}(\pi(p(x)),\pi(p(y))) \leq  d_X(\Phi(\tilde{\pi}(x)),\Phi(\tilde{\pi}(y))) + 2\epsilon.
\]
This completes the proof.
\end{proof}

Now, we turn to the proof of Theorem \ref{TangentialGluing}.

\begin{proof}
Let us consider a sequence $(f_n:M\to(0,1])$ of constant functions on $M$ converging to zero. Obviously, it can be used as a sequence of warping functions. By Theorem \ref{ConvergenceTheorem}, the limit $\lim_{GH} M_{f_n}=\hls{\mathcal{F}}$. We will show, that $\lim_{GH} M_{f_n}=(X,d_X)$.

Let $\{x^1_1,\dots,x^1_{k_1}\}\subset M_1$ and $\{x^2_1,\dots,x^2_{k_2}\}\subset M_2$ be $\epsilon/2$-nets. We can assume that there exist $N\in\mathbb{N}$, $K_1\leq k_1$, and $K_2\leq k_2$ such that
\begin{itemize}
\item $\tilde{\pi}(x^1_i)\neq \tilde{\pi}(x^1_j)$ while $i\neq j$ ($i,j\leq K_1$), $\tilde{\pi}(x^2_l)\neq \tilde{\pi}(x^2_m)$ while $l\neq m$ ($l,m\leq K_2$);
\item $\{x^1_1,\dots,x^1_{K_1}\}$ is an $\epsilon$-net in $(M_1)_{\frac{1}{n}}$, while $\{x^2_1,\dots,x^2_{K_2}\}$ provides an $\epsilon$-net in $(M_2)_{\frac{1}{n}}$, $n>N$.
\end{itemize}
Denote
\begin{itemize}
\item $y_j = \Phi(\tilde{\pi}(x^1_j))$, $j=1,\dots,K_1$;
\item $y_{K_1+j} = \Phi(\tilde{\pi}(x^2_j))$, $j=1,\dots,K_2$;
\item $x_j = \pi(p(x^1_j))$, $j=1,\dots,K_1$;
\item $x_{K_1+j} = \pi(p(x^2_j))$, $j=1,\dots,K_2$.
\end{itemize}
One can easily check that $\{y_1,\dots,y_{K_1+K_2}\}$ and $\{x_1,\dots,x_{K_1+K_2}\}$ have the same number of elements, and $y_i=y_j$ iff $x_i=x_j$. Finally we get two $\epsilon$-nets $\{y_1,\dots,y_K\}$ and $\{x_1,\dots,x_K\}$ in $X$ and $\hls{\mathcal{F}}$ respectively. By Lemma \ref{TangentialGluingLemma}
\[
d_X(x_i,x_j) = \tilde{\rho}(y_i,y_j)
\]
for all $i,j\leq K$. By Lemma \ref{GromovLemma}, $d_{GH}((X,d_X), \hls{\mathcal{F}})=0$. Finally, by Theorem \ref{GromovDistanceTheorem}, we get the statement.
\end{proof}

\begin{rema}
Note that it can be impossible to construct a smooth Riemannian structure $g$ on $M$ such that $g|M_i=g_i$ ($i=1,2$). But all Riemannian structures on a compact manifold are equivalent. We slightly modify the Riemannian structures $g_i$ to obtain structures with desired properties. In this case we can only prove that $\hls{\mathcal{F}}$ of the glued foliation is homeomorphic to $(X,d_X)$.
\end{rema}

\subsection{Transverse gluing}

Following \cite{CC}, let $(M_1,\mathcal{F}_1,g_1)$, $(M_2,\mathcal{F}_1,g_1)$ be smooth compact foliated Riemannian manifolds of dimension $n$ with nonempty boundary and codimension $q$ foliations.  Suppose that $S_i\subset \partial M_i$ is a union of boundary components ($i=1,2$) and $\phi:S_1\to S_2$ is an isometry mapping leaves to leaves. Suppose further that $\mathcal{F}_i$ is $g_i$-orthogonal to $S_i$. Form a manifold $M=M_1\cup_{\phi} M_2$ from the disjoint union $M_1\amalg M_2$ by identifying $x\sim \phi(x)$. Endow $M$ with an induced foliation.

Consider the smallest equivalence relation $\sim$ in disjoint union 
\[
\hls{\mathcal{F}_1}\amalg  \hls{\mathcal{F}_2}
\]
containing the relation defined by
\[
 \tilde{\pi}(x) \sim \tilde{\pi}(\phi(x)).
\] 
Next, glue $\hls{\mathcal{F}_1}$ with $\hls{\mathcal{F}_2}$ along $\sim$ and denote the result endowed with quotient metric by $(X,d_X)$.

\begin{figure}[h]
$$
\xymatrix{
M_1 \amalg M_2 \ar[r]_{p} \ar[d]^{\tilde{\pi}} & M \ar[r]^{\pi} & \hls{\mathcal{F}}\\
\hls{\mathcal{F}_1} \amalg \hls{\mathcal{F}_2}  \ar[r]^{\hskip50pt\Phi}& X &
}
$$
\caption{The projections for Theorem \ref{TransversalGluing}.}
\label{TransversalGluingFigure}
\end{figure}

Let us denote the natural projections as shown on the Figure \ref{TransversalGluingFigure}.

\begin{lemm}\label{TransversalGluingLemma}
For any two points $x\in M_1\amalg M_2$
\[
d_X(\Phi(\tilde{\pi}(x)), \Phi(\tilde{\pi}(y))) = \tilde{\rho}(\pi(p(x)),\pi(p(y))).
\]
\end{lemm}
\begin{proof}
Analogical to the proof of Lemma \ref{TangentialGluingLemma}. Left to the reader.
\end{proof}

\begin{theo}\label{TransversalGluing}
$\hls{\mathcal{F}}$ coincides with $(X,d_X)$.
\end{theo}
\begin{proof}
Denote by $A_1=\{x_1,\dots,x_k\}\subset M_1\setminus \partial M_1$ and $A_2=\{y_1,\dots,y_m\}\subset M_2\setminus \partial M_2$ two $\epsilon$-nets. One can easily check that $p(A_1\cup A_2)$ is an $\epsilon$-net in $M$, $\pi(p(A_1\cup A_2))$ is an $\epsilon$-net in $\hls{\mathcal{F}}$ and $\Phi(\tilde{\pi}(A_1\cup A_2))$ is an $\epsilon$-net in $X$. Moreover, by the construction of $X$ we have that 
\[
\sharp(\Phi(\tilde{pi}\pi(A_1\cup A_2))) = \sharp(\pi(p(A_1\cup A_2))).
\]
Lemma \ref{TransversalGluingLemma} and Lemma \ref{GromovLemma} yield the statement.
\end{proof}

\begin{figure}[h]
 \centering
 \includegraphics{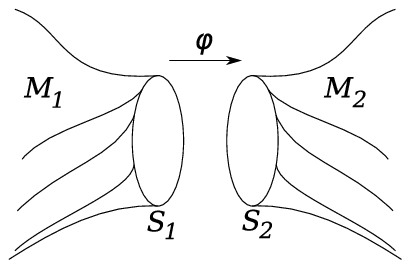}\hskip2cm
 \includegraphics{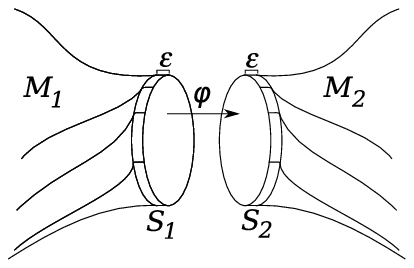}
 \caption{Transverse gluing.}
 \label{TransverseDrawing}
\end{figure}

\begin{rema}
Of course, not every foliation transverse to the boundary component is orthogonal to it. But one can easily modify (see \cite{CC}) any transverse foliation to obtain a foliation orthogonal to the boundary component with the same space as the Hausdorff leaf space (see Figure \ref{TransverseDrawing}). Hence, Theorem \ref{TransversalGluing} is true for any foliations transverse to the boundary.
\end{rema}

\subsection{Turbulization}

Let now $(M,\mathcal{F},g)$  be a foliated Riemannian manifold of dimension $n+1$ endowed with a codimension one foliation which is leaf-wise and transversely orientable. Let $\gamma:[0,1]\to M$ be a closed transversal curve and let $N(\gamma)$ be a fixed foliated tubular neighbourhood of $\gamma$. Let us equip  $N(\gamma)=D^n\times S^1$ with cylindrical coordinates $(r,z,t)$ (we take $t$ {\em modulo} $1$, and the leaves of $\mathcal{F}|N(\gamma)$ are the sets $D^n\times \{t\}$). Let 
\[
\omega = \cos \lambda(r) {\rm d}r + \cos \lambda(r) {\rm d}t,
\]
where $\lambda:[0,1]\to [-\pi/2,\pi/2]$ is a smooth, strictly increasing on $[0,3/4]$ function satisfying $\lambda(0)=-\pi/2$, $\lambda(2/3)=0$, $\lambda(t)=\pi/2$ for all $t\geq 3/4$, and with derivatives of all orders at zero vanishing. Since $\omega$ is integrable, it defines a foliation $\mathcal{F}_{\gamma}$ of $N(\gamma)$, which agrees with $\mathcal{F}$ near $\partial N(\gamma)$ and has a Reeb component $R$ inside $N(\gamma)$. Modified foliation $\mathcal{F}_\gamma$ of $M$ is called a {\it turbulized} foliation, while this deformation is called {\it turbulization} \cite{CC} (Figure \ref{TurbulizationFigure1}).

\begin{figure}[h]
 \centering
 \includegraphics{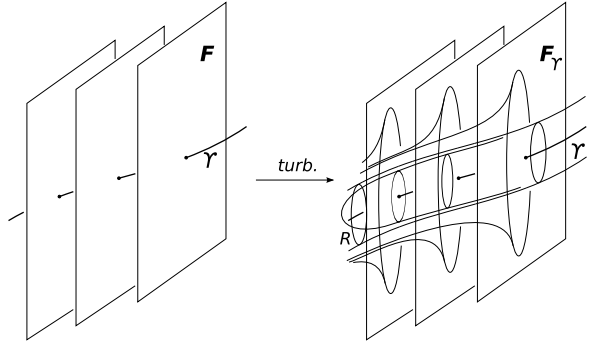}
 \caption{Turbulization.}
 \label{TurbulizationFigure1}
\end{figure}

Denote by $L_x$ ($L^{\gamma}_x$) the leaf of $\mathcal{F}$ ($\mathcal{F}_{\gamma}$) passing through $x\in M$. Next, let $\pi:M\to \hls{\mathcal{F}}$ be a natural projection, and let $X$ be a metric space obtained from $\hls{\mathcal{F}}$ by identification $\pi(\gamma([0,1]))$ to a point (Figure \ref{TurbulizationFigure2}). Equip $X$ with the quotient metric $d_X$.

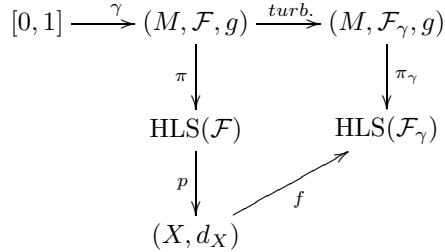
\begin{figure}
$$
\xymatrix{
[0,1] \ar[r]^{\gamma} & (M,\mathcal{F},g) \ar[d]_{\pi} \ar[r]^{turb.} & (M,\mathcal{F}_{\gamma},g) \ar[d]^{\pi_{\gamma}}\\
 & \hls{\mathcal{F}} \ar[d]_{p} & \hls{\mathcal{F}_{\gamma}}\\
 & (X,d_X) \ar[ru]_{f} & 
}
$$
\caption{The projections for Theorem \ref{Turbulization}.}
\end{figure}

\begin{theo}\label{Turbulization}
$\hls{\mathcal{F}_{\gamma}}$ is isometric with $(X, d_X)$.
\end{theo}

\begin{figure}[h]
 \centering
 \includegraphics{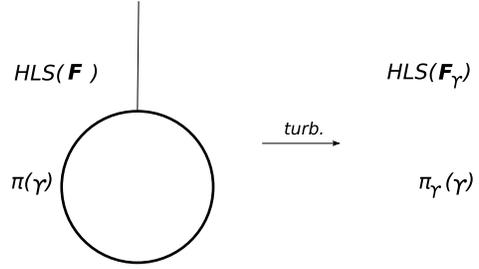}
 \caption{HLS for turbulized foliation.}
 \label{TurbulizationFigure2}
\end{figure}

Before we start a proof, we shall prove technical lemmas.

\begin{lemm}\label{TurbulizationLemma1}
For every two leaves $L_1,L_2\in\mathcal{F}$ and every $\epsilon>0$ there exists a finite sequence of leaves $F_1,\dots,F_k \in\mathcal{F}_{\gamma}$, $k\leq 3$, satisfying
\begin{enumerate}
 \item $F_1\setminus N(\gamma) = L_1\setminus N(\gamma)$ and $F_k\setminus N(\gamma) = L_2\setminus N(\gamma)$,
 \item $\sum_{\nu=1}^{k-1} \dist(F_{\nu},F_{\nu+1}) \leq \dist (L_1,L_2) + \epsilon$.
\end{enumerate}
\end{lemm}
\begin{proof}
Let $\epsilon>0$, and $x\in L_1$, $y\in L_2$ be such that $d(x,y)\leq \dist(L_1,L_2)+\epsilon$. We shall consider three cases:
\begin{enumerate}
 \item $x,y\notin M\setminus N(\gamma)$. Put $F_1=L^{\gamma}_x$ and $F_2=L^{\gamma}_y$. Then $ \dist(F_1,F_2)\leq d(x,y)\leq\dist(L_1,L_2)+\epsilon$. 
 \item $x,y\in N(\gamma)$. By the definition of the turbulization we have that $\dist(L^{\gamma}_x,L^{\gamma}_y)=0\leq \dist (L_1,L_2)$.
 \item $x\in N(\gamma)$ and $y\in M\setminus N(\gamma)$. Let $z\in L_1\setminus N(\gamma)$. Put $F_1=L^{\gamma}_z$, $F_2=L^{\gamma}_x$, $F_3=L^{\gamma}_y$. We have
\[
\dist(F_1,F_2) + \dist(F_1,F_2) \leq 0 + d(x,y) \leq \dist(L_1,L_2)+\epsilon.
\]
\end{enumerate}
This completes the proof.
\end{proof}

\begin{lemm}\label{TurbulizationLemma2}
For any $x,y\in M$ the following inequality holds:
\[
d_X(p(\pi(x)),p(\pi(y)))\leq \dist(L^{\gamma}_x, L^{\gamma}_x).
\]
\end{lemm}
\begin{proof}
Let $x,y\in M$. We shall consider few cases:

\noindent {\bf Case 1.} $L^{\gamma}_x\cap N(\gamma) = L^{\gamma}_y\cap N(\gamma) = \emptyset$. Then $L_x = L^{\gamma}_x$, $L_y = L^{\gamma}_y$, and 
\[
d_X(p(\pi(x)),p(\pi(y)))\leq \dist(L^{\gamma}_x, L^{\gamma}_x).
\]

\noindent {\bf Case 2.} $L^{\gamma}_x\cap N(\gamma)\neq \emptyset$, $L^{\gamma}_x\cap N(\gamma)\neq\emptyset$. Then, by the construction of $\mathcal{F}_{\gamma}$,
\[
L_x\cap N(\gamma)\neq \emptyset,\quad {\rm and}\quad L_y\cap N(\gamma)\neq \emptyset.
\]
Finally,
\[
d_X(p(\pi(x)),p(\pi(y)))=0\leq \dist(L^{\gamma}_x, L^{\gamma}_x).
\]

\noindent {\bf Case 3.} $L^{\gamma}_x\cap N(\gamma)\neq \emptyset$, $L^{\gamma}_x\cap N(\gamma)=\emptyset$. Let $\epsilon>0$, and $r\in L^{\gamma}_x$, $q\in L^{\gamma}_y$ be such points that $d(r,q)\leq \dist(L^{\gamma}_x,L^{\gamma}_y) + \epsilon$.

Suppose first that $x\notin N(\gamma)$, $r\notin N(\gamma)$. Then $L_x=L_r$, $L_y=L_q$, and
\begin{align*}
d_X(p(\pi(x)),p(\pi(y))) \leq \dist(L_x,L_y) \leq d(r,q)\\
\leq \dist(L^{\gamma}_x, L^{\gamma}_x).
\end{align*}
Next, suppose that $x\notin N(\gamma)$, $r\in N(\gamma)$. Set $L_1=L_x$, $L_2=L_r$, $L_3=L_y$. Recall that $L_q=L_y$. Hence, by Case 1,
\begin{align*}
d_X(p(\pi(x)),p(\pi(y))) \leq d_X(p(\pi(x)),p(\pi(r))) + d_X(p(\pi(r)),p(\pi(q)))\\
\leq 0+ d(r,q) \leq \dist(L^{\gamma}_x, L^{\gamma}_x)+\epsilon.
\end{align*}
Analogically we show that 
\[
d_X(p(\pi(x)),p(\pi(y))) \leq \dist(L^{\gamma}_x, L^{\gamma}_x),
\]
for $x\in N(\gamma)$, $r\notin N(\gamma)$, and $x\in N(\gamma)$, $r\in N(\gamma)$.

Passing with $\epsilon$ to zero gives the desired inequality.
\end{proof}

\begin{lemm}\label{TurbulizationLemma3}
For any $x,y\in M$ we have
\[
d_X(p(\pi(x)),p(\pi(y)))\leq \tilde{\rho}_\gamma(\pi_{\gamma}(L^{\gamma}_x), \pi_{\gamma}(L^{\gamma}_x)).
\]
\end{lemm}
\begin{proof}
Let $\epsilon>0$, $x,y\in M$. There exists a sequence of leaves $L^{\gamma}_1,\dots, L^{\gamma}_k$ such that $L^{\gamma}_1=L^{\gamma}_x$, $L^{\gamma}_k=L^{\gamma}_y$, and 
\[
\sum_{\nu=1}^{k=1} \dist(L^{\gamma}_{\nu},L^{\gamma}_{\nu+1}) \leq \tilde{\rho}_{\gamma}(\pi_{\gamma}(L^{\gamma}_x), \pi_{\gamma}(L^{\gamma}_x))+\epsilon.
\]
Let $r_1,q_1,\dots,r_{k-1},q_{k-1}\in M$ be such that $r_i\in L^{\gamma}_i$, $q_i\in L^{\gamma}_{i+1}$, and
\[
d(r_i,q_i) \leq \dist(L^{\gamma}_{\nu},L^{\gamma}_{\nu+1})  + \frac{\epsilon}{k}.
\]
Note that $L^{\gamma}_x=L^{\gamma}_{r_1}$, $L^{\gamma}_y=L^{\gamma}_{q_{k-1}}$, and $L^{\gamma}_{r_{i+1}} = L^{\gamma}_{q_i}$. By Lemma \ref{TurbulizationLemma2},
\begin{align*}
d_X(p(\pi(x)),p(\pi(r_1)))=0,\\
 d_X(p(\pi(y)),p(\pi(q_{k-1})))=0,\\
d_X(p(\pi(r_{i+1})),p(\pi(q_i)))=0
\end{align*}
for all $i\in\{1,\dots,k-1\}$. By the construction of $X$,
\begin{align*}
d_X(p(\pi(x)),p(\pi(y)))\leq \\
\sum_{\nu=1}^{k-1} d_X(p(\pi(r_{i})),p(\pi(q_i))) + \sum_{\nu=1}^{k-1} d_X(p(\pi(r_{i+1})),p(\pi(q_i)))\\
+ d_X(p(\pi(x)),p(\pi(r_1))) + d_X(p(\pi(y)),p(\pi(q_{k-1})))\\
\leq \sum_{\nu=1}^{k-1} d(r_i,q_i)
\leq \tilde{\rho}_\gamma(\pi_{\gamma}(L^{\gamma}_x), \pi_{\gamma}(L^{\gamma}_x))+2\epsilon.
\end{align*}
Passing with $\epsilon$ to zero gives us the statement.
\end{proof}

\begin{lemm}\label{TurbulizationLemma4}
For any $x,y\in M$ we have
\[
\tilde{\rho}_\gamma(\pi_{\gamma}(L^{\gamma}_x), \pi_{\gamma}(L^{\gamma}_x))\leq d_X(p(\pi(x)),p(\pi(y))).
\]
\end{lemm}
\begin{proof}
Let $x,y\in M$, $\epsilon>0$. There exist points $r_1,q_1,\dots,r_k,q_k\in \hls{\mathcal{F}}$ such that $p(q_i)=p(r_{i+1})$, $\pi(x)=r_1$, $\pi(x)=q_k$, and
\begin{equation}\label{TurbLemma4Eq1}
\sum_{\nu=1}^{k-1} \tilde{\rho}(r_i,q_i)\leq d_X(p(\pi(x)),p(\pi(y)))+\epsilon.
\end{equation}
For any $i\in\{1,\dots,k\}$ one can find a family of leaves $L_{i,1},\dots,L_{i,\mu_i}$ satisfying $L_{i+1,1}=L_{i,\mu_i}$, $x\in L_{1,1}$, $y\in L_{k,\mu_k}$, and
\begin{equation}\label{TurbLemma4Eq2}
\sum_{\nu=1}^{\mu_i-1} \dist (L_{i,\nu},L_{i,\nu+1}) \leq \tilde{\rho}(r_i,q_i)+\frac{\epsilon}{k}.
\end{equation}
By (\ref{TurbLemma4Eq1}), (\ref{TurbLemma4Eq2}), and Lemma \ref{TurbulizationLemma1}, one can find a finite sequence $L^{\gamma}_1,\dots, L^{\gamma}_m$ of leaves of $\mathcal{F}_{\gamma}$ such that $L^{\gamma}_1=L^{\gamma}_x$, $L^{\gamma}_m=L^{\gamma}_y$ and
\begin{align*}
\tilde{\rho}_\gamma(\pi_{\gamma}(L^{\gamma}_x), \pi_{\gamma}(L^{\gamma}_x))\leq \sum_{\nu=1}^{m-1} \dist(L^{\gamma}_{\nu},L^{\gamma}_{\nu+1})\\
\leq d_X(p(\pi(x)),p(\pi(y))) + 3\epsilon.
\end{align*}
Passing with $\epsilon$ to zero gives us the statement.
\end{proof}

Let $f:X\to\hls{\mathcal{F}_{\gamma}}$ be defined as follows:
\begin{equation}\label{TurbulizationIsometry}
f(p(\pi(x))) = \pi_{\gamma}(L^{\gamma}_x),
\end{equation}
where $\pi_{\gamma}:M\to\hls{\mathcal{F}_{\gamma}}$ denotes the natural projection.

\begin{lemm}\label{TurbulizationLemma5}
$f$ is bijective. 
\end{lemm}
\begin{proof}
Suppose that $p(\pi(x)) = p(\pi(y))$, $x,y\in M$. Consider two cases: 

\noindent {\bf Case 1.} $\pi(x)=\pi(y)$. If $\pi^{-1}(\pi(x))\cap \gamma([0,1]) = \emptyset$ then $\pi^{-1}(\pi(y))\cap \gamma([0,1]) = \emptyset$, and $\pi_{\gamma}^{-1}(\pi_{\gamma}(x))=\pi_{\gamma}^{-1}(\pi_{\gamma}(y))$. Hence $\pi_{\gamma}(L^{\gamma}_x) = \pi_{\gamma}(L^{\gamma}_y)$, and $f(p(\pi(x))) = f(p(\pi(y)))$. 

\noindent {\bf Case 2.} If $\pi(x)\neq\pi(y)$, then there exist $\xi_x\in \pi^{-1}(\pi(x))\cap \gamma([0,1])$ and $\xi_y\in \pi^{-1}(\pi(y))\cap \gamma([0,1])$. Hence, $\pi_{\gamma}(\xi_x) = \pi_{\gamma}(\xi_x)=\pi_{\gamma}(R)$, where $R$ denotes the Reeb component. But $\pi_{\gamma}(L^{\gamma}_x) = \pi_{\gamma}(L^{\gamma}_{\xi_x})$. Thus $f(p(\pi(x))) = f(p(\pi(x)))$. 

Finally, $f$ is well defined. By the definition, $f$ is ''onto'' $\hls{\mathcal{F}_{\gamma}}$. Checking that $f$ is one-to-one we leave to the reader.
\end{proof}

Now, we can turn to the proof of Theorem \ref{Turbulization}.

\begin{proof}
By Lemma \ref{TurbulizationLemma5}, $f$ defined in (\ref{TurbulizationIsometry}) is a bijection from $X$ onto $\hls{\mathcal{F}_{\gamma}}$. By Lemmas \ref{TurbulizationLemma3} and \ref{TurbulizationLemma4}, $f$ is an isometry.
\end{proof}

\subsection{Spinning}

Following the definition given in \cite{CC} we recall the notion of {\em spinning} a foliation along a transverse boundary component.

Let $(M,\mathcal{F},g)$ be a compact Riemannian manifold carrying codim-1 foliation transverse to the boundary $\partial M\neq \emptyset$. Let $S$ be a transverse connected component of $\partial M$ with $\partial S=\emptyset$. Assume that $\mathcal{F}|S$ can be defined by a closed non-singular $1$-form $\omega\in A^1(S)$.

\begin{figure}[h]
 \centering
 \includegraphics[scale=1]{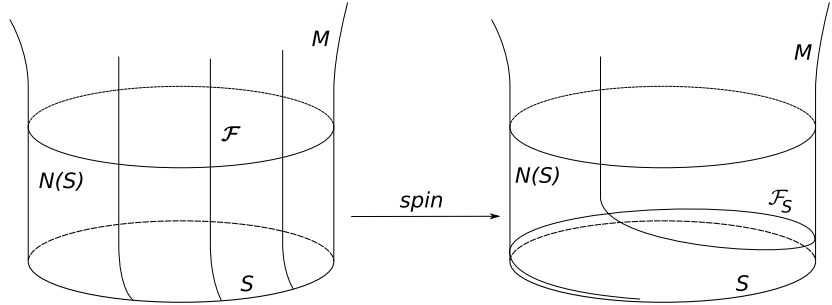} 
 \caption{Spinning along the boundary component.}
 \label{SpinningFoliationFigure1}
\end{figure}

Let $N(S) = S\times [0,1)$ be a foliated collar, i.e. the leaves of $\mathcal{F}|N(S)$ are of the form $L\times [0,1)$, where $L$ is a leaf of $\mathcal{F}|S$.

Decompose $T_{(x,t)}(N(S)) = T_x(S)\oplus T_t([0,1))$. One can write a vector field $\zeta\in X(N(S))$ as
\[
\zeta = fv+g\partial_t,
\]
where $f,g\in C^{\infty}(N(S))$, $v\in X(S)$, and $\partial_t = \frac{\partial}{\partial t}$. $\omega$ extends to a closed non-singular form $\omega_{N(S)}$ by
\[
\omega_{N(S)}(fv+g\partial_t) = f\omega(v).
\]

Let $h:[0,1)\to [0,1]$ be a $C^{\infty}$-function such that $h(t)=0$ for $t\in[\frac{1}{2},1)$, $h(0)=1$, and $h$ is decreasing strictly monotonically on $[0,\frac{1}{2}]$. Moreover, let the derivatives of all orders of $h$ vanish at $t=0$. Set 
\[
\theta = (1-h(t))\omega_{N(S)} + h(t){\rm d}t.
\]
$\theta$ agrees with $\omega_{N(S)}$ on $S\times [\frac{1}{2},1)$ and with ${\rm d}t$ on $S\times \{0\}$. Moreover, $\theta$ is integrable and $S$ becomes a leaf of a new foliation $\mathcal{F}_S$ on $S\times [0,1)$. But $\mathcal{F}$ coincides with $\mathcal{F}_S$ outside the collar $S\times [0,\frac{1}{2})$. Thus, we extend $\mathcal{F}_S$ to a foliation $\mathcal{F}_S$ on $M$ which is tangent to the boundary component $S$.

Now, identify in $\hls{\mathcal{F}}$ the points of $\pi(S)$ and denote the result by $X$. Endow $X$ with the quotient metric denoted by $d_X$.

Before we examine the Hausdorff leaf space for a spinned foliation we formulate technical lemmas. Easy proofs are omitted and left to the reader.

\begin{figure}[h]
$$\xymatrix{
 M \ar[r]^{\pi} \ar[d]^{\pi_S} & \hls{\mathcal{F}}\ar[d]_{\phi}\\
\hls{\mathcal{F}_S} & X
}
$$
\caption{The projections for Theorem \ref{Spinning}.}
\label{SpinningFigure}
\end{figure}

Let $\pi:M\to \hls{\mathcal{F}}$, $\pi_S:M\to \hls{\mathcal{F}_S}$, and $\phi:\hls{\mathcal{F}}\to X$ denote the natural projections (Figure \ref{SpinningFigure}). Denote by $L_z$ ($L^S_z$) a leaf of $\mathcal{F}$ ($\mathcal{F}_S$) passing through a point $z\in M$.

\begin{lemm}\label{SpinLemma1}
For every two points $p,q\in M$ such that $L^S_p=L^S_q$ we have
\[
d_X(\phi(\pi(L_p)),\phi(\pi(L_q)))=0. 
\]
\hfill $\square$
\end{lemm}

\begin{lemm}\label{SpinLemma2}
For any two points $p,q\in M$ such that $L_p=L_q$ we have 
\[
\tilde{\rho}(\pi_S(L^S_p), \pi_S(L^S_q))=0.
\]
\hfill $\square$
\end{lemm}

\begin{lemm}\label{SpinLemma3}
For any two points $x,y\in M$ we have
\[
d_X(\phi(\pi(L_x)),\phi(\pi(L_y))) = \tilde{\rho}_S(\pi_S(x),\pi_S(y)).
\]
\hfill $\square$
\end{lemm}

\begin{theo}\label{Spinning}
$\hls{\mathcal{F}}$ coincides with $(X,d_X)$.
\end{theo}

\begin{figure}[h]
 \centering
 \includegraphics[scale=1]{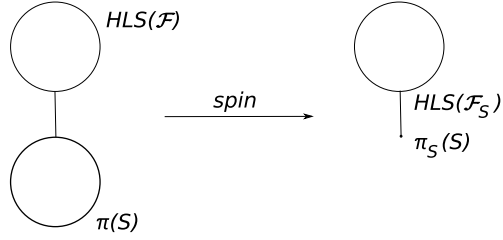} 
 \caption{HLS of a foliation spinned along the boundary component $S$.}
 \label{SpinningFoliationFigure2}
\end{figure}

\begin{proof}
Let $f_n = \frac{1}{n}$ be a constant function on $M$. Let $A'=\{x_1,\dots, x_{k'}\}\subset M$ be an $\epsilon/2$-net on $M$. One can select from a subset $A=\{x_1,\dots, x_k\}\subset A'$ and $N\in\mathbb{N}$ such that $\pi(x_i)\neq \pi(x_j)$ ($i\neq j$) and $A$ an $\epsilon$-net on $M_{\frac{1}{n}}=(M,\mathcal{F},g_{\frac{1}{n}})$ for all $n>N$. We may assume that the points $x_{k-l},\dots,x_k$ are the only ones that belong to $\pi^{-1}(\pi(S))$. Now, pick from the points $x_{k-l},\dots,x_k$ exactly one, let say $x_{k-l}$. 

Observe that $\pi_S(\{x_{k-l},\dots,x_k\})$ is a single point in $\hls{\mathcal{F}_S}$. Hence, there exists $N'$ such that $\{x_1,\dots,x_{k-l}\}$ is an $\epsilon$-net on $M^S_{\frac{1}{n}}=(M,\mathcal{F}_S,g_{\frac{1}{n}})$ for all $n>N'$. Moreover, since $x_{k-l},\dots,x_k\in \pi^{-1}(\pi(S))$, 
\[
 \phi(\pi(x_{\mu})) = \phi(\pi(x_{\nu})),\quad \mu,\nu\in\{k-l,\dots,k\}.
\]
Set $\zeta_i=\phi(\pi(x_i))$, $\xi_j=\pi_S(x_j)$ ($i,j=1,\dots,k-l$). By the construction and Lemma \ref{ConvergenceTheoremLemma1}, the sets $\{\zeta_i\}$ and $\{\xi_j\}$ are $2\epsilon$-nets on $X$ and $\hls{\mathcal{F}_S}$, respectively. By Lemma \ref{SpinLemma3},
\[
d_X(\zeta_i,\zeta_j)=\tilde{\rho}_S(\xi_i,\xi_j), \quad \textrm{for all $i,j\in\{1,\dots,k\}$}.
\]
By Lemma \ref{GromovLemma}, $d_{GH}(X,\hls{\mathcal{F}_S})=0$, and by Theorem \ref{GromovDistanceTheorem}, $X$ is isometric to $\hls{\mathcal{F_S}}$.
\end{proof}

\subsection{Suspension}

Denote by $B$ a smooth connected manifold, and by $p:\tilde B\to B$ the universal covering of $B$. Let $x_0\in B$. Recall that the covering transformation group $\Gamma$ acts from the right on $\tilde B$ and hence $\Gamma\subset {\rm Diff} (\tilde B)$. Let $F$ be a $q$-dimensional manifold. Consider a group homomorphism $h:\Gamma\to{\rm Diff}(F)$. Then $\Gamma$ acts on $\tilde B\times F$ by
\[
\gamma(x,z) = (\gamma(x), h(\gamma)(z)), \quad (x\in \tilde B, z\in F).
\]
Consider a foliation $\tilde F = \{\tilde B\times \{z\} , z\in F\}$. Using canonical projection one can project $\tilde F$ onto a foliation $\mathcal{F}$ of $M=(\tilde B\times F)/\Gamma$. The foliation $\mathcal{F}$ is called the {\em suspension} of the homomorphism $h$. One can check that $M$ is a fibre bundle over $B$, and $F$ coincides with its fibre.

Analogically as in Section \ref{Preliminaries-HLS}, one can define the {\it Hausdorff orbit space}:

Let $G$ be a group acting on a metric space $(X,d_X)$. Denote by $\mathcal{O}$ the space of orbits of $G$-action. Set
\begin{equation*}
\rho(G(x),G(y)) = \inf \{\sum_{i=1}^{n-1} d_X (G_1,G_{i+1})\},
\end{equation*}
where the infimum is taken over all finite sequences of orbits beginning at $G_1=G(x)$ and ending at $L_n=G(y)$, and $G(z)$ denotes the orbit of $z\in X$. Define an equivalence relation $\sim$ in $\mathcal{O}$  by:
\[
G(x)\sim G(y) \Leftrightarrow \rho(G(x),G(y))=0,\quad x,y\in X.
\]
Let $\tilde{\mathcal{O}} = \mathcal{O}/_{\sim}$. Put 
\[
\tilde\rho([G(x)],[G(y)]) = \rho(G(x),G(y)), 
\]
where $[G(x)],[G(y)]\in\tilde{\mathcal{O}}$. $(\tilde{\mathcal{O}},\tilde\rho)$ is a metric space. We call it the {\em Hausdorff orbit space} of the $G$-action, and we denote it by $\HOS{X/G}$.

\begin{theo}
$\hls{\mathcal{F}}$ is homeomorphic to $\HOS{F/h(\Gamma)}$.
\end{theo}
\begin{proof}
By the construction of suspension, there exists a homeomorphism between the space of leaves of $\mathcal{F}$ and the space of orbits of $h(\Gamma)$. It induces a homeomorphism between $\hls{\mathcal{F}}$ and $\HOS{F/h(\Gamma)}$.
\end{proof}

\section{Main results - HLS for codim-1 foliations}\label{MainResults}

\subsection{HLS for compact I-bundles}

Let $(M,\mathcal{F},{\rm pr})$ be a foliated $I$-bundle, $I=[0,1]$. Note that there are at most two boundary leaves. Let us denote by $L_0$ the boundary leaf passing through the points $0\in I$ of every fiber. Consider the function $d:\mathcal{L}\to [0,1]$ ($\mathcal{L}$ denotes here the space of leaves of the foliation $\mathcal{F}$) defined by $d(L)=\tilde{\rho}(L_0,L)$, where $\tilde{\rho}$ denotes the metric in $\hls{\mathcal{F}}$. Let $\pi:M\to \hls{\mathcal{F}}$ again be the natural projection.

\begin{lemm}\label{Codim-1-lm0}
For any two leaves $L\neq L'$ such that $\pi(L)\neq \pi(L')$ we have $d(L)\neq d(L')$.
\end{lemm}
\begin{proof}
Since $\pi(L)\neq \pi(L')$ then $\tilde{\rho}(\pi(L),\pi(L'))>0$. Let $\epsilon>0$, and let $L_1,\dots,L_k$ be a family of leaves such that $L_k=L'$, and 
\[
\sum_{\nu=0}^{k-1} \dist(L_{\nu},L_{\nu+1})< \tilde{\rho}(L_0,L').
\]
Without losing generality we can assume that there exists $j\in\{0,\dots,k-1\}$ satisfying $L_j=L$ (if not then rename the leaf $L$ to $L'$ and $L'$ to $L$). Then
\begin{align*}
\tilde{\rho}(L_0,L) + \tilde{\rho}(L,L') \leq \sum_{\nu=0}^{j-1} \dist(L_{\nu},L_{\nu+1}) + \sum_{\nu=j}^{k-1} \dist(L_{\nu},L_{\nu+1})\\
< \tilde{\rho}(L_0,L') + \epsilon.
\end{align*}
Hence,
\[
d(L)+\tilde{\rho}(L,L') \leq d(L').
\]
By the triangle inequality and the above, we obtain
\[
d(L)+\tilde{\rho}(L,L') = d(L').
\]
But $\tilde{\rho}(L,L')>0$. Hence, $d(L)<d(L')$. This completes the proof.
\end{proof}

\begin{theo}\label{Codim-1-lm1}
Let $(M,\mathcal{F},{\rm pr})$ be a foliated $I$-bundle. $\hls{\mathcal{F}}$ is isometric to a metric segment.
\end{theo}
\begin{proof}
Let $L_0$ denote the same leaf as in Lemma \ref{Codim-1-lm0}, $d$ be a function on the space of leaves of $\mathcal{F}$  defined by $d(L)=\tilde{\rho}(L_0,L)$, and let $\delta=\max_{L\in\mathcal{F}} d(L)$. Let $\pi:M\to\hls{\mathcal{F}}$ be a natural projection, while $p:M\to [0,\delta]$ be the mapping defined by $p(x)=d(L_x)$. By Lemma \ref{Codim-1-lm0}, for any two leaves such that $\pi(L)\neq \pi(L')$ we have 
\[
d(L)\neq d(L').
\]

Let $\epsilon>0$, and $L,L'\in\mathcal{F}$ be two arbitrary leaves such that $d(L)<d(L')$. Let  $L_1,\dots,L_k,L_{k+1},\dots,L_{k+l}$ be a family of leaves satisfying $L_k=L$, $L_{k+l}=L'$, and
\[
\sum_{\nu=0}^{k+l-1} \dist(L_{\nu},L_{\nu+1})\leq \tilde\rho(\pi(L_0),\pi(L'))+\epsilon. 
\]
Since $\tilde{\rho}(\pi(L_0),\pi(L))\leq \sum_{\nu=0}^{k-1} \dist(L_{\nu},L_{\nu+1})$, we have
\begin{align}\label{I-bundleEq0}
\tilde{\rho}(\pi(L),\pi(L'))\leq \sum_{\nu=0}^{k+l-1} \dist(L_{\nu},L_{\nu+1}) \\
\leq \tilde{\rho}(\pi(L_0),\pi(L')) +\epsilon - \tilde{\rho}(\pi(L_0),\pi(L)) = |d(L')-d(L)| + \epsilon.
\nonumber
\end{align}
Now, let $L_1,\dots,L_k, L_{k+1},\dots,L_{k+l}$ be a family of leaves such that $L_k=L$, $L_{k+l}=L'$
\[
\sum_{\nu=0}^{k-1} \dist(L_{\nu},L_{\nu+1}) \leq \tilde{\rho}(\pi(L_0),\pi(L))+\frac{\epsilon}{2},
\]
and 
\[
\sum_{\nu=k}^{k+l-1} \dist(L_{\nu},L_{\nu+1}) \leq \tilde{\rho}(\pi(L),\pi(L'))+\frac{\epsilon}{2}.
\]
Then
\begin{align*}
d(L') \leq \sum_{\nu=0}^{k+l-1} \dist(L_{\nu},L_{\nu+1}) \leq \tilde{\rho}(\pi(L_0),\pi(L))+\frac{\epsilon}{2} +\tilde{\rho}(\pi(L),\pi(L'))+\frac{\epsilon}{2}\\
\leq d(L) + \tilde{\rho}(\pi(L),\pi(L'))+\epsilon.
\end{align*}
We get
\begin{equation}\label{I-bundleEq1}
d(L') - d(L) \leq \tilde{\rho}(\pi(L),\pi(L'))+\epsilon.
\end{equation}
Since $d(L)-d(L')\leq 0 \leq \tilde{\rho}(\pi(L),\pi(L'))$ we finally get, by (\ref{I-bundleEq0}) and (\ref{I-bundleEq1}),
\begin{equation}\label{I-bundleEq2}
||d(L)-d(L')| - \tilde{\rho}(\pi(L),\pi(L'))| \leq\epsilon.
\end{equation}

Let $A=\{x_1,\dots,x_k\}$ be an $\epsilon$-net on $M$. Then $\pi(A)$ and $p(A)$ are $\epsilon$-nets on $\hls{\mathcal{F}}$ and $([0,d],|\cdot|)$, respectively. Moreover, $\sharp\pi(A)=\sharp p(A)$. By (\ref{I-bundleEq2}), we have
\[
||p(L_i)-p(L_j)| - \tilde{\rho}(\pi(L_i),\pi(L_j))| \leq\epsilon,
\]
where $L_{\nu}=L_{x_{\nu}}$.
By Lemma \ref{GromovLemma}, $d_{GH}(\hls{\mathcal{F}},[0,d])\leq 3\epsilon$. Finally, 
\[d_{GH}(\hls{\mathcal{F}},[0,d])=0,
\]
and, by Theorem \ref{GromovDistanceTheorem}, $\hls{\mathcal{F}}$ is isometric to the metric segment $I=([0,d],|\cdot|)$.
\end{proof}

\subsection{HLS for codim-1 foliations}

Recall now \cite{BBI} that the metric graph $G$ is the result of gluing of a set of a disjoint metric segments $E=\{E_i\}$ and points $V=\{v_i\}$ along an equivalence relation defined in the union of $V$ and the set of the endpoints of the segments equipped with the length metric. A graph $G$ is called {\em finite} if $V$ and $E$ are finite.

\begin{theo}\label{MainTheorem1}
$\hls{\mathcal{F}}$ of any codimension one foliation on a compact Riemannian manifold is isometric to a finite connected metric graph.
\end{theo}
\begin{proof}
Following the proof of the main theorem of \cite{IW}, we can cover $M$ by a finite number of mutually disjoint saturated  neighborhoods $N_i$ ($i=1,\dots,k$) such that the HLS of the foliation restricted to $N_i$ is a singleton, and a finite number of mutually disjoint foliated $I-bundles$ (denoted by $C_1,\dots,C_m$) with their HLS's, by Lemma \ref{Codim-1-lm1}, isometric to $[0,d_j]$, $d_j>0$, $1\leq j\leq m$. We can assume that $N_i\cap C_j\subset \partial N_i \cap \partial C_j$, $1\leq i\leq k$, $1\leq j\leq m$ (Figure \ref{GraphConstrFigure1.1}), and that the sets $N_i$ ($i=1,2,\dots,k$) are maximal, i.e. $\pi^{-1}(\pi(N_i))=N_i$, where $\pi:M\to \hls{\mathcal{F}}$ denotes the natural projection.

\begin{figure}[h]
\centering
 \includegraphics[scale=1]{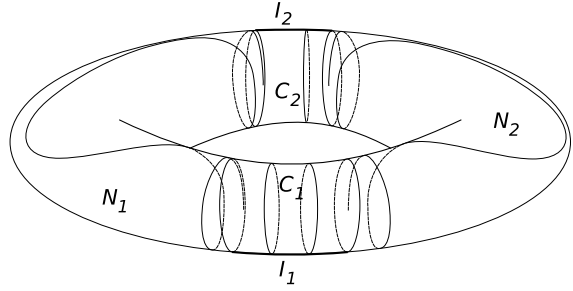}
 \caption{The sets $N_i$ and $C_j$.}
 \label{GraphConstrFigure1.1}
\end{figure}

Let $v_i= \hls{\mathcal{F}|N_i}$, and $V=\{v_1,\dots,v_k\}$. Next, let 
\[
E=\{I_1,\dots,I_m\},\quad I_j=\hls{\mathcal{F}|C_j}=[0,d_j].
\]

Denote by $\pi_j:C_j\to [o,d_j]$ natural projections.

Introduce in $V$ and in the set of the endpoints of the segments $I_j$, $1\leq j\leq m$, the smallest equivalence relation $\sim$ generated by the following relation:

A point $v_i$ is in the relation with an endpoint $a$ ($a$ can be equal to $0$ or $d_j$) of the segment $I_j$ iff $N_i\cap \pi_j^{-1}(a)\neq \emptyset$.

\begin{figure}[h]
\centering
 \includegraphics[scale=1]{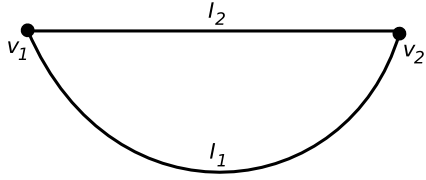}
 \caption{Construction of a graph.}
 \label{GraphConstrFigure1.2}
\end{figure}

Glue points from $V$ and segments from $E$ along $\sim$. Obtained space endow with the length metric. In this way we obtain a metric graph $G$ (Figure \ref{GraphConstrFigure1.2}). By the construction of $G$ and Theorem \ref{TangentialGluing}, $\hls{\mathcal{F}}$ is isometric to $G$.
\end{proof}

\begin{rema}
One can easily check that it is possible to construct a number of metric graphs, not necessarily finite, isometric to $\hls{\mathcal{F}}$, but all of them are isometric as metric spaces wit length metric. 
\begin{figure}[h]
\centering
 \includegraphics[scale=1]{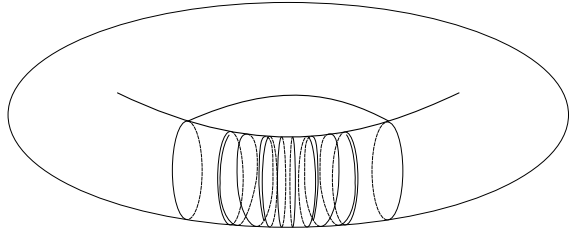}
 \caption{A part of a foliation by Kronecker components and circles.}
 \label{KroneckerFoliation}
\end{figure}
For example, consider a foliation of $T^2$ by a infinite number of Kronecker components separated by circles foliation (Figure \ref{KroneckerFoliation}). Then every Kronecker component can define itself a node of a graph, and every circle foliation can define an edge. One also can select only one Kronecker component to be a node, and the rest of foliation to be an edge. One can check that any metric graph constructed this way is isometric to a circle.
\end{rema}

\begin{exem}
Recall that any compact manifold of dimension $1$ is either an interval $I$ or a $1$-dimensional sphere $S^1$. Hence, a foliated bundle of codim-1 is either $I$-bundle or $S^1$-bundle. One can see that Hausdorff leaf space for a codim-1 foliated bundle is a singleton, a metric segment or a circle $S^1$.
\end{exem}

\begin{lemm}\label{Codim-1-lm21}
For every $k\in\mathbb{N}$ there exists a compact foliated manifold $(M,\mathcal{F})$ such that $M$ has exactly $k$ boundary components and $\hls{\mathcal{F}}$ is a singleton, and the holonomy mappings $h$ of the boundary leaves satisfy $h(0)=0$, $h'(0)=1$, $h^{(n)}(0)=0$ for all $n\geq 2$.
\end{lemm}
\begin{proof}
Let $\hat M=S^1\times \Sigma$, where $\Sigma$ is a compact surface of dimension $2$, and let $\hat{\mathcal{F}}$ be the product foliation by $\{z\}\times \Sigma$, $z\in S^1$. Let $x_1,\dots, x_k\in S^2$. Let $N_i$ ($i=1,\dots,k$) be disjoint tubular neighbourhoods of $\gamma_i=S_1\times \{x_i\}$. Turbulize $\hat{\mathcal{F}}$ simultaneously along $\gamma_i$. One can check \cite{CC} that it is possible to turbulize in such way that the holonomy mappings $h$ of the compact leaves of the Reeb components satisfy $h(0)=0$, $h'(0)=1$, $h^{(n)}(0)=0$ for all $n\geq 2$.

Next, let $M$ be a foliated manifold obtained from $(M,\mathcal{F})$ by removing the interior of the Reeb components of the turbulized foliation. It follows that $M$ is compact, and its boundary has exactly $k$ components homeomorphic with the torus $T^2$. Moreover, every leaf different from boundary leaves accumulate on every boundary component. Thus $\hls{\mathcal{F}}$ is a singleton, and $\mathcal{F}$ is a foliation with desired properties.
\end{proof}

\begin{rema}\label{remark:codim1Proper}
One can see that all leaves of the foliation constructed in Lemma \ref{Codim-1-lm21} are proper.
\end{rema}

\begin{lemm}\label{Codim-1-lm22}
For any metric segment $I=[0,d]$ there exists a compact foliated Riemannian manifold $(M,\mathcal{F},g)$ carrying codim-1 foliation such that $\hls{\mathcal{F}}$ is isometric to $I$.
\end{lemm}
\begin{proof}
Taking $M=[0,d]\times \Sigma$, where again $\Sigma$ is a compact surface, with product foliation $\{t\}\times \Sigma$ and the product metric we get the statement. 
\end{proof}

\begin{theo}\label{MainTheorem2}
For every finite connected metric graph $G$ there exists a compact foliated Riemannian manifold $(M,\mathcal{F},g)$ such that $\hls{\mathcal{F}}$ is isometric to $G$. Moreover, every leaf of $\mathcal{F}$ is proper.
\end{theo}
\begin{proof}
Let $G=(V,E)$ be a finite connected metric graph with $k$ nodes. ''Cutting'' every edge in the middle we obtain $k$ connected metric graphs $G_i$ (Figure \ref{GraphConstrFigure2.1}).

\begin{figure}[h]
\centering
 \includegraphics[scale=1]{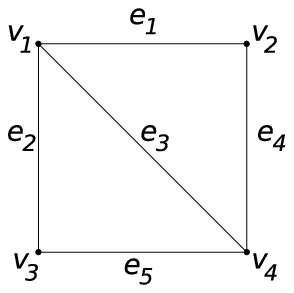}\hskip3cm \includegraphics[scale=1]{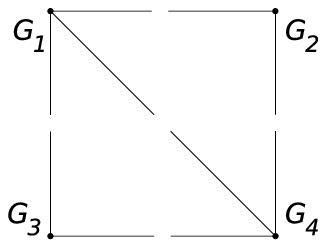}
 \caption{Star graphs $G_i$.}
 \label{GraphConstrFigure2.1}
\end{figure}

Consider a graph $G_i$. If all nodes of $G_i$ have only one edge, then assign for $G_i$ a foliated manifold indicated in Lemma \ref{Codim-1-lm22}. 

Let $v$ be a node having more than one edge, let say $m$. One can assign for $v$ a $3$-dimensional foliated Riemannian manifold $(V_i,\mathcal{F}_i,g_i)$ indicated in Lemma \ref{Codim-1-lm21} with exactly $m$ boundary components homeomorphic to the torus $T^2$, and such that HLS for $V_i$ is a singleton, and the holonomy mappings $h$ of the boundary leaves satisfy $h(0)=0$, $h'(0)=1$, $h^{(n)}(0)=0$ for all $n\geq 2$. 

Next, for every edge assign a manifold $E^i_{\nu}=[0,d_i]\times T^2$ (as described in Lemma \ref{Codim-1-lm22}), $1\leq \nu\leq m$. Note that either $\mathcal{F}_i$ or foliations of $E_i$ are tangent to the boundary components. 

Since the holonomy mappings $h$ of the boundary leaves satisfy $h(0)=0$, $h'(0)=1$, $h^{(n)}(0)=0$ for all $n\geq 2$, then by Theorem \ref{TangentialGluing}, one can glue manifolds $V_i$ and $E_i$ to obtain a compact foliated Riemannian manifold $(M_i,\mathcal{F}_i,g_i)$ with $\hls{\mathcal{F}_i}$ isometric to $G_i$ (Figure \ref{GraphConstrFigure2.2}).
\begin{figure}[h]
\centering
 \includegraphics[scale=1]{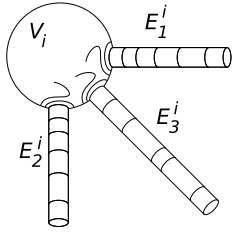}
 \caption{Construction of a manifold $M_i$ for the graph $G_i$.}
 \label{GraphConstrFigure2.2}
\end{figure}
Moreover, the boundary components of $M_i$ ($i=1,\dots,m$) homeomorphic to $T^2$, foliations $\mathcal{F}_i$ on each $M_i$ are tangent to the boundary, and holonomy mappings $h$ of boundary leaves satisfy $h'(0)=1$, $h^{(n)}(0)=0$ for all $n\geq 2$.

\begin{figure}[h]
\centering
 \includegraphics[scale=1]{graph-constr2.1.eps}
 \hskip2cm
 \includegraphics[scale=1]{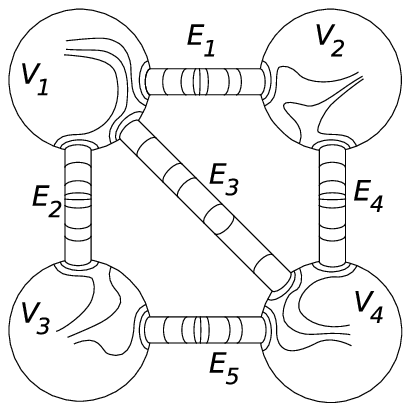}
 \caption{The graph $G$ and the manifold $(M,\mathcal{F},g)$.}
 \label{GraphConstrFigure2.3}
\end{figure}

Again, by Theorem \ref{TangentialGluing}, one can glue manifolds $M_i$ to get a compact foliated manifold $(M,\mathcal{F},g)$ such that $\hls{\mathcal{F}}$ is isometric to $G$ (Figure \ref{GraphConstrFigure2.3}).

By Remark \ref{remark:codim1Proper}, all leaves of $\mathcal{F}$ are proper. This ends our proof.
\end{proof}

\subsection{Warped foliations in codim-1}

Let $(M,\mathcal{F},g)$ be an arbitrary compact foliated Riemannian manifold with codim-1 foliation. Let $(f_n)_{n\in\mathbb{N}}$, $f_n:M\to (0,1]$, be a sequence of warping functions (see Section \ref{Preliminaries-WarpedFoliations}). We will now develop the necessary and sufficient condition for a sequence of warped foliations $(M_{f_n})_{n\in\mathbb{N}}$ to converge to the Hausdorff leaf space for the foliation $\mathcal{F}$.

First, note that on any connected finite metric graph $G$ with at least two nodes there exist a measure $\mu$ constants and $\beta\geq 1$, $\eta_0>0$ such that for all $\eta<\eta_0$ and $x\in G$
\begin{equation}\label{BishopEq}
\frac{1}{\beta} \eta\leq \mu(B_d(x,\eta)) \leq \beta\eta,
\end{equation}
where $B_d(x,\eta) = \{y\in X:\quad d(x,y)<\eta\}$. Indeed, denote by $E=\{e_1,\dots,e_k\}$ the set of vertices, and by $V=\{I_1,\dots,I_m\}$ the set of all edges of the graph $G$. Let $\mu$ be a measure induced by the Lebesgue measure on edges $I_j$ of $G$ and let $\eta_0= \frac{1}{2}\min l(I_j)$ and $\beta=\max\{2, \max_{i=1,\dots,k} n(e_i)\}$, where $l(I)$ denotes the length of an edge $I$, and $n(e)$ denotes the number of edges in a vertex $e$. Such $\mu$ satisfies (\ref{BishopEq}).

Let $(f_n)_{n\in\mathbb{N}}$, $f_n:M\to (0,1]$, be a sequence of warping functions on $(M,\mathcal{F},g)$, where $\mathcal{F}$ is a foliation of codimension one.

\begin{theo}\label{SufficientNecessaryCodim1}
$d_{GH}((M,g_{f_n}),\hls{\mathcal{F}})\to 0$ if and only if for every $\varepsilon>0$ there exists $N\in\mathbb{N}$ such that for any $n>N$  the following is satisfied:

There exists a finite family of leaves $F^n=\{F^n_1,\dots,F^n_k\}$ such that
\begin{enumerate}
\item $\bigcup F^n$ is $\varepsilon$-dense in $M$,
\item $f_n|_{\bigcup F^n} < \varepsilon$. 
\end{enumerate}
\end{theo}

The proof of the sufficient condition is analogical to the proof of Theorem \ref{ConvergenceTheorem} in Section \ref{Convergence}. The proof of the necessary condition is the same as the proof of Theorem 6.5 in \cite{W}. We don't here repeat them and we left them for the reader. 

\section{Final remarks}\label{FinalRemarks}

One can ask, what is the classification of HLS for foliations of codimension greater than one. This question still is open. We only present some results for an arbitrary codimension. 

Let $(M,\mathcal{F},g)$ be a compact connected foliated Riemannian manifold, and again let $\pi:M\to\hls{\mathcal{F}}$ be the natural projection. One can easily check that $\pi$ is continuous. Moreover, for any leaf $L\in\mathcal{F}$ the set $\pi^{-1}(\pi(L))$ is a closed, nonempty, saturated subset of $M$.

Let us recall that a subset $A\subseteq M$ is called minimal if it is nonempty, closed and saturated and there is no proper subset of $A$ with these properties \cite{CC}. From the construction of $\hls{\mathcal{F}}$ it follows that for any leaf $L\in\mathcal{F}$ the set $\pi^{-1}(\pi(L))$ contains a minimal set.

As a simple consequence of the above observations we have:

\begin{theo}\label{MinimalSetsTheorem}
If the number of minimal sets of $\mathcal{F}$ is countable then the $\hls{\mathcal{F}}$ is a singleton.
\end{theo}
\begin{proof}
Since the number of minimal sets is countable, then $\hls{\mathcal{F}}$ is a countable set. The projection $\pi:M\to\hls{\mathcal{F}}$ is continuous, hence $\hls{\mathcal{F}}$ is compact and connected. This ends our proof.
\end{proof}

\begin{theo}
If $\mathcal{F}$ contains a compact leaf with finite holonomy then $\hls{\mathcal{F}}$ contains an open subset $U$ homeomorphic to an open set of $\mathbb{R}^q$, where $q$ is a codimension of $\mathcal{F}$.
\end{theo}
\begin{proof}
This is a direct consequence of the Reeb Stability Theorem (see \cite{CC} or \cite{HH}).
\end{proof}

\end{document}